\documentclass{amsart}[12pt]
\usepackage{amssymb,amsmath}
\usepackage{enumerate}
\usepackage[english]{babel}
\usepackage{color}

\newtheorem{teo}{Theorem}[section]
\newtheorem{pro}{Proposition}[section]
\newtheorem{cor}{Corollary}[section]

\theoremstyle{definition}

\title{Isometric $F$-spaces of $log$-integrable  function}

\begin{document}
\date{September 14, 2019}

\begin{abstract}
Let $(\Omega_i,\mathcal A_i,\mu_i) $ be a measure space with finite measure  $\mu_i$,  and let $(L_{\log}(\Omega_i, \mathcal A_i,\mu_i), \|\cdot\|_{\log,\mu_i})$  be a $F$-space  of all $\log$-integrable  functions on $(\Omega_i,\mathcal A_i,\mu_1), \ i =1, 2 $. It is proved that $F$-spaces  $(L_{\log}(\Omega_1, \mathcal A_1,\mu_1), \|\cdot\|_{\log,\mu_1})$ \and \ $(L_{\log}(\Omega_2, \mathcal A_2,\mu_2), \|\cdot\|_{\log,\mu_2})$ are  isometric if and only if there exists a measure preserving isomorphism from  $(\Omega_1, \mathcal A_1,\mu_1)$  onto  $(\Omega_2, \mathcal A_2,\mu_2)$.
\end{abstract}

\author{R.Z. Abdullaev }
\address{Tashkent University of Information Technologies\\
Tashkent,  100200, Uzbekistan}
\email{arustambay@yandex.com}
\author{V.I. Chilin}
\address{National University of Uzbekistan\\
Tashkent, 100174, Uzbekistan}
\email{vladimirchil@gmail.com; chilin@ucd.uz}
\author{B.A. Madaminov}
\address{Urgench state Unversity\\
Urgench, 220100, Uzbekistan}
\email{aabekzod@mail.ru}

\maketitle

\section{Introduction}

The study of linear isometries of Banach function spaces was begun by S. Banach \cite{B},
who gave a description of all linear isometries for the spaces $ L_p [0,1]$. \
J. Lamperti received a description of all isometries  $U$ from $L_{p}(\Omega_1, \mathcal A_1, \mu_1)$ into  $L_{p}(\Omega_2, \mathcal A_2, \mu_2)$, \ $1\leq p < \infty, \ p \neq 2$, where $(\Omega_i, \mathcal A_i,\mu_i)$ is an arbitrary measure space with the finite measure, $ i=1, 2$ \ \cite{L} (see also \cite[Ch. 3, \S 3.2, Theorem 3.2.5]{FJ}). \
One of the corollaries of such descriptions  is the following
\begin{teo}\label{t11}
Let  $(\Omega_i, \mathcal A_i, \mu_i)$ be a  measure space with finite measure, \ let $\nabla_i$ be a complete Boolean algebra of all  classes  of equal $\mu_i$-almost everywhere sets from a $\sigma$-algebra $\mathcal A_i$,  $ i=1, 2$, \  $1\leq p < \infty, \ p \neq 2$. Then a Banach spaces  \ $L_{p}(\Omega_1, \mathcal A_1, \mu_1)$ \ and \ $L_{p}(\Omega_2, \mathcal A_2, \mu_2)$ \ are  isometric if and only if \ there \ exists \ a \ Boolean isomorphism \ $\varphi:\nabla_1 \to \nabla_2$.
\end{teo}

An important metrizable analogue of Banach spaces $L_{p}(\Omega, \mathcal A, \mu)$ are the $F$-space  $L_{\log}(\Omega, \mathcal A,\mu)$ \ of all $\log$-integrable  functions
introduced in \cite{DSZ}. The $F$-space $L_{\log}(\Omega, \mathcal A,\mu)$ is defined by the equality
$$ L_{\log}(\Omega, \mathcal A, \mu)=\{f \in L_0(\Omega, \mathcal A, \mu): \|f\|_{\log,\mu}= \int \limits_{\Omega} \log(1+|f|)d\mu< + \infty\},$$
where  $ L_0(\Omega, \mathcal A, \mu)$ is the algebra of equivalence classes of almost everywhere (a.e.) finite real (complex) valued measurable functions on  $(\Omega, \mathcal A, \mu)$.
It is known that $L_{log}(\Omega, \mathcal A, \mu)$  is a subalgebra in algebra
$ L_0(\Omega, \mathcal A, \mu)$, in addition, a $F$-norm
$$
\|f\|_{\log, \mu}=\int_{\Omega} \log(1+|f|)d\mu, \ f \in L_{log}(\Omega, \mathcal A, \mu)
$$
is defined a metric $\rho_{log,\mu}(f,g)=\|f-g\|_{\log,\mu}$ on $L_{\log}(\Omega, \mathcal A, \mu)$ such that the pair $(L_{log}(\Omega, \mathcal A, \mu), \rho_{log})$ is a complete metric topological vector space \ \cite{DSZ}.

Naturally the problem arises to find the necessary and sufficient
conditions ensuring  isometric  of the  $F$-spaces  $L_{\log}(\Omega_1, \mathcal A_1, \mu_1)$ and $L_{\log}(\Omega_2, \mathcal A_2, \mu_2)$. In Section \ref{s3} we give the following criterion for the existence of a surjective isometry for $F$-spaces of \ $log$-integrable  functions.
\begin{teo}\label{t12}
Let  $(\Omega_i, \mathcal A_i, \mu_i)$ be a  measure space with finite measure, $ i=1, 2$. \ Then the $F$-spaces \ $L_{\log}(\Omega_1, \mathcal A_1, \mu_1)$ and $L_{\log}(\Omega_2, \mathcal A_2, \mu_2)$ are isometric if and only there exists a preserves measures a Boolean isomorphism $\varphi:(\nabla_1, \mu_1) \to (\nabla_2, \mu_2)$.
\end{teo}

\section{Preliminaries} \label{s2}

Let $(\Omega, \mathcal A, \mu)$ be a  measure space with finite measure $\mu$, and let  $L_0(\Omega, \mathcal A, \mu) $  (respectively, $ L_{\infty}(\Omega, \mathcal A, \mu)$) be the algebra  of equivalence classes of  real (complex) valued measurable functions (respectively, essentially bounded real (complex) valued measurable functions)  on $(\Omega, \mathcal A, \mu)$. Let $\nabla$ the complete Boolean algebra of all  classes  $[A]$  of equal $\mu$-a.e. sets $A \in \mathcal A$. It is known that $\widehat{\mu}([A])=\mu(A)$ is a strictly positive finite measure
on $\nabla.$  Below, we  denote the measure $\widehat {\mu}$ by $\mu,$  the algebra $L_0(\Omega, \mathcal A, \mu)$ (respectively, $ L_{\infty}(\Omega, \mathcal A, \mu)$) by
$L_0(\nabla)$ (respectively, $L_{\infty}(\nabla)$ and the integral $\int\limits_{\Omega} f d\mu$ by $\int\limits_{\nabla} f d\mu$.

Following to \cite{DSZ}, we consider in $L_0(\nabla)$ the subalgebra
$$
L_{\log}(\nabla, \mu)=\{f \in L_0(\nabla): \int\limits_{\nabla} \log(1+|f|)d\mu< + \infty\} $$
of $\log$-integrable measurable functions, and for each $f \in L_{\log}(\nabla, \mu),$ we set
$$
\|f\|_{\log}=\int\limits_{\nabla} \log(1+|f|)d\mu.
$$
A non-negative function
$\|\cdot\|_{\log, \mu}: L_{\log}(\nabla, \mu) \rightarrow [0,\infty)$
is a $F$-norm on the linear space $L_{\log}(\nabla, \mu)$, that is  (see, for example, \cite[Ch. 1, \S \ 2]{K}),

$(i)$. $\|f\|_{\log, \mu}>0$ for all $0 \neq f \in L_{\log}(\nabla, \mu);$

$(ii)$. $\|\alpha f\|_{\log, \mu}\le\|f\|_{\log}, \mu$  for all $f \in
L_{\log}(\nabla, \mu)$ and  number $\alpha$ with $|\alpha|\leq
1;$

$(iii)$. $\lim_{\alpha\to 0}\|\alpha f\|_{\log, \mu}=0$ for all $f \in
L_{\log}(\nabla, \mu);$

$(iv)$. $\|f+g\|_{\log, \mu}\leq\|f\|_{\log, \mu}+\|g\|_{\log, \mu}$ for all $f,
g \in L_{\log}(\nabla, \mu)$.

It is known that  the set  $L_{\log}(\nabla, \mu)$  is a complete metric topological vector space with respect to the metric  $\rho(f,g)=\|f-g\|_{\log, \mu}$ \ \cite{DSZ}, in addition, \ $L_{p}(\nabla, \mu)\subset L_{\log}(\nabla, \mu)$ \ for all $1\leq p < \infty$.

Note also that if $f_n \in L_{\log}(\nabla, \mu), \ |f_n| \leq g \in L_{\log}(\nabla, \mu), \ n =1,2,\dots$, and $f_n \to f \in L_0(\nabla)$ \ $\mu$-a.e. then $f \in L_{\log}(\nabla, \mu)$ \ and \ $\|f_n-f\|_{\log} \to 0$ as $n \to \infty$.

Let $(\nabla_1, \mu_1), \  (\nabla_2, \mu_2)$ \ be complete Boolean algebras with a strictly positive finite measures, and let $\varphi:\nabla_1 \to \nabla_2$ be a Boolean isomorphism.  By Theorem \cite[Theorem 2.3]{chl}  there exists a unique  isomorphism $\Phi$ \ from algebra \ $L_{0}(\nabla_1)$ \ onto algebra \ $L_{0}(\nabla_2)$ \ such that $\Phi(e)=\varphi(e)$ for all $e \in  \nabla_1$. \ It is clear that
the function \ $\lambda(\varphi(e)), \ e \in \nabla_1,$
is a strictly positive finite measure on Boolean algebra $\nabla_2$, in addition, \ $\Phi(L_{\log}(\nabla_1, \mu_1))= L_{\log}(\nabla_2, \lambda)$ \ \cite[Proposition 3]{ach}.

Denote by  \ $\frac{d\lambda}{d\mu_2}$ \  the Radon-Nikodym  derivative of measure $\lambda$ with respect to the measure $\mu_2$. It is well known that  $0 \leq \frac{d\lambda}{d\mu_2} \in L_{0}(\nabla_2)$, \ and \
$f \in  L_1(\nabla_2, \lambda)$ \ if and only if \ $\big(f\cdot \frac{d\lambda}{d\mu_2}\big) \in L_1(\nabla, \mu_2)$, \ in addition, $\int\limits_{\nabla_2} f \ d\lambda =\int\limits_{\nabla_2} f \cdot(\frac{d\lambda}{d\mu_2}) \ d\mu_2.$

 Using Proposition \cite[Proposition 3]{ach} we get the following
\begin{pro}\label{p21}
$\Phi(L_{\log}(\nabla_1, \mu_1)) = L_{\log}(\nabla_2, \mu_2) \ \Longleftrightarrow \ \frac{d\lambda}{d\mu_2}, \ \frac{d\mu_2}{d\lambda} \in L_{\infty}(\nabla_2)$.
\end{pro}
If an isomorphism $\varphi:\nabla_1 \to \nabla_2$ preserves a measure, that is, $ \mu_2(\varphi(e))= \mu_1(e)$ for all $e \in \nabla_1$, then using the equality $\Phi(\log(1+|f|)) = \log(1+\Phi(|f|)), \ f \in \mathcal L_0(\nabla_1)$ \ \cite[Proposition 3]{ach} we get
$$
 \|\Phi(f)\|_{\log,\mu_2}=\int\limits_{\nabla_2} \log(1+|\Phi(f)|)d\mu_2=\int\limits_{\nabla_2} \Phi(\log(1+|f|)) d\mu_2 =
$$
$$=\int\limits_{\nabla_1} \log(1+|f|) d\mu_1=  \|f\|_{\log,\mu_1} \ \ \text{for all} \ \ f \in L_{\log}(\nabla_1, \mu_1).
$$
Therefore, using Proposition \ref{p21}, we give the following
\begin{teo}\label{t22}
If $\varphi:\nabla_1 \to \nabla_2$ is a measure-preserving isomorphism  then the algebraic isomorphism \ $\Phi: L_{\log}(\nabla_1, \mu_1) \to L_{\log}(\nabla_2, \mu_2)$ is a surjective positive linear  isometry.
\end{teo}

In the next section, we  will give a  full description of all linear isometries from $L_{\log}(\nabla_1, \mu_1)$ \ into \ $L_{\log}(\nabla_2, \mu_2)$.

\section{Isometries  of $F$-spaces of $log$-integrable  functions}\label{s3}

In this section we give a   description of all linear isometries  from \ $L_{\log}(\nabla_1, \mu_1)$ \ into \ $L_{\log}(\nabla_2, \mu_2)$.

We need the following "disjointness" \ property of   linear isometries  from $L_{\log}(\nabla_1, \mu_1)$ \ into \ $L_{\log}(\nabla_2, \mu_2)$.
\begin{pro}\label{p31}
Let \ $U:L_{\log}(\nabla_1, \mu_1) \to L_{\log}(\nabla_2, \mu_2)$ \ be \ a \ linear isometry. Then for any \ $f, \ g \in  L_{\log}(\nabla_1, \mu_1)$ \ such that \ $f\cdot g =0$ \
the equality \ $U(f)\cdot U(g) =0$ \ is true.
\end{pro}
\begin{proof}
Let $e=s(f)= \mathbf 1_{\nabla_1}- \sup\{p \in \nabla_1 : p\cdot f=0\}$ \ and \ $q=s(g)$ \ be the supports of functions $f, \ g \in  L_{\log}(\nabla_1, \mu_1))$, \ let $f_1 = U(f), \ g_1 = U(g)$. \  Then  $ e\cdot q= 0$ \ and \
$$\int\limits_{\nabla_2}\log(1+|f_1+g_1|) \ d \mu_2  = \|f_1+g_1\|_{\log,\mu_2}=\|U(f+g)\|_{\log,\mu_2} = \|f+g\|_{\log,\mu_1}=
$$
$$
=\int\limits_{\nabla_1}\log(1+|f+g|) \ d \mu_1= \ (\text{since} \ e\cdot q= 0) \ = \int\limits_{\nabla_1}\log(1+|f|+|g|) \ d \mu_1=
 $$
 $$=\int\limits_{e\cdot \nabla_1}\log(1+|f|) \ d \mu_1 +\int\limits_{q\cdot \nabla_1}\log(1+|g|) \ d \mu_1=
$$
$$
\|f\|_{\log,\mu_1}+\|g\|_{\log,\mu_1}=\|f_1\|_{\log,\mu_2}+\|g_1\|_{\log,\mu_2}=
$$
$$
=\int\limits_{\nabla_2}\log(1+|f_1|) \ d \mu_2 +\int\limits_{\nabla_2}\log(1+|g_1|) \ d \mu_2=\int\limits_{\nabla_2}\log(1+|f_1|+|g_1|+ |f_1|\cdot |g_1|) \ d \mu_2,
$$
what is impossible in the case when $\mu_2(\{|f_1|\cdot |g_1|\})>0$. Consequently, $|f_1|\cdot |g_1|=0$, in particular, $U(f)\cdot U(g) =0$.
\end{proof}
Using Proposition \ref{p31} we get the following  full description of  linear isometries  $U:L_{\log}(\nabla_1, \mu_1) \to L_{\log}(\nabla_2, \mu_2)$.
\begin{teo}\label{t32} (cf. \cite[Theorem 3.6]{HSZ}).
Let $U:L_{log}(\nabla_1, \mu_1) \to L_{\log}(\nabla_2, \mu_2)$ be a   linear isometry. Then there exists an injective $\sigma$-additive homomorphism $\Phi$ \ from algebra \ $L_{0}(\nabla_1)$ \ into algebra \ $L_{0}(\nabla_2)$  \ such that
$$U(f) = U(\mathbf 1_{\nabla_1})\cdot \Phi(f)\ \ \text{for each} \ \ f\in L_{\log}(\nabla_1, \mu_1).$$
In addition,
$$ |U(\mathbf 1_{\nabla_1})|= -1 +2\cdot \frac{d \lambda}{d  \mu_2},  \eqno (1) $$
where $\lambda(\Phi(e)) = \mu_1(e), \ e \in \nabla_1$.
\end{teo}
\begin{proof}
We set $\varphi(e) = s(U(e)), \ e \in \nabla_1$, \ $p_1= \varphi(\mathbf 1_{\nabla_1})$. As shown in \cite[Theorem 3.2.5]{FJ}, the map $\varphi$  is an injective $\sigma$-additive Boolean homomorphism from $\nabla_1$ \ into \ $p_1\cdot \nabla_2$. Using proof of Theorem 2.3 \cite{chl} we get that there exists a unique injective $\sigma$-additive homomorphism $\Phi$ \ from algebra \ $L_{0}(\nabla_1)$ \ into algebra \ $L_{0}(\nabla_2)$  such that $\Phi(e) = \varphi(e)$ for all $e \in \nabla_1$. In addition, the restriction  $J = \Phi|_{ L_{\infty}(\nabla_1)}$ is a $\|\cdot\|_{\infty}$-continuous injective homomorphism from $L_{\infty}(\nabla_1)$ into  $L_{\infty}(\nabla_2)$. Since
$$
U(e) = U(\mathbf 1_{\nabla_1}-(\mathbf 1_{\nabla_1}-e))=U(\mathbf 1_{\nabla_1}-(\mathbf 1_{\nabla_1}-e))\cdot s(U(e))=
$$
$$
=U(\mathbf 1_{\nabla_1})\cdot \varphi(e)-U(\mathbf 1_{\nabla_1}-e)\cdot \varphi(\mathbf 1_{\nabla_1}-e)\cdot \varphi(e)=U(\mathbf 1_{\nabla_1})\cdot J(e).
$$
and homomorphism $J$ is a $\|\cdot\|_{\infty}$-continuous it follows that
$$U(f)=U(\mathbf 1_{\nabla_1})\cdot J(f)=U(\mathbf 1_{\nabla_1})\cdot \Phi(f) \ \ \text{for all} \ \ f \in L_{\infty}(\nabla_1).  \eqno (2)$$
We show now that the equality $(1)$ is true for all $f\in L_{\log}(\nabla_1, \mu_1)$. It suffices to verify equality $(2)$ for any  $0\leq f\in L_{\log}(\nabla_1, \mu_1)$. Choose a sequence $0\leq f_n\in L_{\infty}(\nabla_1)$ such that $f_n \uparrow f$. Then
$$
\|U(\mathbf 1_{\nabla_1})\cdot \Phi(f_n)-U(f)\|_{\log,\mu_2}=\|U(f_n)-U(f)\|_{\log,\mu_2}=
$$
$$
\|U(f_n-f)\|_{\log,\mu_2}=\|f_n-f\|_{\log,\mu_1} \to 0 \ \text{as} \ n \to \infty  \eqno(3)
$$
Since \ $\Phi(f_n) \uparrow \Phi(f), \ \ \Phi(f_n) \in  L_{\infty}(\nabla_2)$, \ it follows that \ $U(\mathbf 1_{\nabla_1})\cdot \Phi(f_n) \in L_{\log}(\nabla_2, \mu_2)$ for all $n=1,2,\dots$, \ $U(\mathbf 1_{\nabla_1})\cdot \Phi(f_n) \to  U(\mathbf 1_{\nabla_1})\cdot \Phi(f)$ $\mu_2$-a.e., in addition, $\{U(\mathbf 1_{\nabla_1})\cdot \Phi(f_n)\}_{n=1}^\infty$ is a Cauchy sequence in $F$-space \ $(L_{\log}(\nabla_2,\mu_2), \|\cdot\|_{\log,\mu_2})$. Using the completeness of  $(L_{\log}(\nabla_2,\mu_2), \|\cdot\|_{\log,\mu_2})$, we obtain that
$$\|U(\mathbf 1_{\nabla_1})\cdot \Phi(f_n)-g\|_{\log,\mu_2} \to 0$$
for some $g \in L_{\log}(\nabla_2,\mu_2)$.
By $(3)$ we have that
$$U(f)=g = (\mu_2-a.e.)\lim\limits_{n \to \infty} U(\mathbf 1_{\nabla_1})\cdot \Phi(f_n)=U(\mathbf 1_{\nabla_1})\cdot \Phi(f).$$
We show now that $ |U(\mathbf 1_{\nabla_1})|= -1 +2\cdot \frac{d \lambda}{d  \mu_2}$. It is clear that $\nu(e) = \mu_2(\Phi(e)), \ e \in \nabla_1$, is a strictly positive finite measure on Boolean algebra $\nabla_1$. Since
$$
\int\limits_{\nabla_2} \Phi(f) d \mu_2 = \int\limits_{\nabla_1} f d \nu \ \ \text{for all} \ \ f \in L_{1}(\nabla_1, \nu),
$$
it follows for each $e \in \nabla_1$ that
$$
\int\limits_{e\cdot \nabla_1 }\log 2 \  d \mu_1=\|e\|_{\log,\mu_1}= \|U(e)\|_{\log,\mu_2} = \int\limits_{\nabla_2}\log\big(1+ |U(\mathbf 1_{\nabla_1})|\cdot\Phi(e)\big) \  d \mu_2=
$$
$$
=\int\limits_{\nabla_2}\Phi \big(\log\big(1+ \Phi^{-1} \big(|U(\mathbf 1_{\nabla_1}|) \big)\cdot e \big ) \  d \mu_2=\int\limits_{\nabla_1}\log\big(1+ \Phi^{-1} \big(|U(\mathbf 1_{\nabla_1})| \big)\cdot e \big) \  d \nu=
$$
$$
=\int\limits_{e\cdot \nabla_1}\log\big(1+ \Phi^{-1} \big(|U(\mathbf 1_{\nabla_1}|) \big)\cdot \frac{d \nu}{d  \mu_1} \  d \mu_1.
$$
Consequently, $$ \log 2=\log\big(1+ \Phi^{-1} \big(|U(\mathbf 1_{\nabla_1}|) \big)\cdot \frac{d \nu}{d  \mu_1},$$
that is,
$$
2 = (1+ \Phi^{-1} \big(|U(\mathbf 1_{\nabla_1}) \big|)\cdot \frac{d \nu}{d  \mu_1}.
$$
Thus
$$
\Phi^{-1} \big(|U(\mathbf 1_{\nabla_1}|)\big) = -1+ 2\cdot \big(\frac{d \nu}{d  \mu_1}\big)^{-1}=-1+ 2\cdot \frac{d  \mu_1}{d  \nu},
$$
and
$$
|U(\mathbf 1_{\nabla_1})| = -1+ 2\cdot \Phi \big(\frac{d  \mu_1}{d  \nu}\big).
$$
Next, if \ $e \in \nabla_1$ \ then
$$
\int\limits_{\nabla_2}\Phi \big(\frac{d  \mu_1}{d  \nu}\big)\cdot \Phi(e) \ d  \mu_2 = \int\limits_{\nabla_2}\Phi \big(\frac{d  \mu_1}{d  \nu}\cdot e \big)\ d  \mu_2=\int\limits_{\nabla_1}\frac{d  \mu_1}{d  \nu}\cdot e \ d  \nu= \mu_1(e).
$$
Consequently,
$$
\int\limits_{\Phi(e)\cdot \nabla_2}\frac{d  \lambda}{d  \mu_2} \ d \mu_2=\int\limits_{\nabla_2}\frac{d  \lambda}{d  \mu_2}\cdot \Phi(e) \ d  \mu_2=\lambda(\Phi(e)) = \mu_1(e) =
$$
$$
=\int\limits_{\nabla_2}\Phi \big(\frac{d  \mu_1}{d  \nu}\big) \cdot \Phi(e) \ d  \mu_2= \int\limits_{\Phi(e)\cdot \nabla_2}\Phi \big(\frac{d  \mu_1}{d  \nu}\big) \ d  \mu_2.
$$
Thus $\frac{d  \lambda}{d  \mu_2}= \Phi \big(\frac{d  \mu_1}{d  \nu}\big)$, and \
$|U(\mathbf 1_{\nabla_1})| = -1+ 2\cdot \frac{d  \lambda}{d  \mu_2}.$
\end{proof}
Remark that in  \cite{HSZ}  the version of Theorem \ref{t32} without $(1)$ is obtained \ for any $F$-spaces and for positive isometries with "disjointness" \ property.

The following Corollary refines Theorem 1 for surjective linear isometries.
\begin{cor}\label{c33}  (cf. Theorem \ref{t11}).
Let $U:L_{log}(\nabla_1, \mu_1) \to L_{\log}(\nabla_2, \mu_2)$ be a  surjective linear isometry.  Then there exists an isomorphism $\Phi$ \ from algebra \ $L_{0}(\nabla_1)$ \ onto algebra \ $L_{0}(\nabla_2)$  \ such that \
$$U(f) = U(\mathbf 1_{\nabla_1})\cdot \Phi(f) \ \ \text{for each} \ \ f\in L_{\log}(\nabla_1, \mu_1),$$
in particular, $\Phi(L_{log}(\nabla_1, \mu_1)) = L_{\log}(\nabla_2, \mu_2)$.
\end{cor}
\begin{proof}
Since $U^{-1}:L_{log}(\nabla_2, \mu_2) \to L_{\log}(\nabla_1, \mu_1)$ is a  linear isometry it follows from Theorem \ref{t32} that there exists an injective $\sigma$-additive homomorphism \ $\Psi$ \ from algebra \ $L_{0}(\nabla_1)$ \ into algebra \ $L_{0}(\nabla_2)$  \ such that \ $U^{-1}(g) = U^{-1}(\mathbf 1_{\nabla_2})\cdot \Psi(h)$ for each $g\in L_{\log}(\nabla_2, \mu_2)$. Therefore
$$
f = U^{-1}(U(f)) =  U^{-1}(U(\mathbf 1_{\nabla_1})\cdot \Phi(f))= U^{-1}(\mathbf 1_{\nabla_2})\cdot \Psi(U(\mathbf 1_{\nabla_1}))\cdot \Psi(\Phi(f))
$$
for all $f \in L_{\log}(\nabla_1, \mu_1)$.
In particular, $\mathbf 1_{\nabla_1}=U^{-1}(\mathbf 1_{\nabla_2})\cdot \Psi(U(\mathbf 1_{\nabla_1}))$. Thus $f =\Psi(\Phi(f))$ for each $f\in L_{\infty}(\nabla_1)$. This means that $\Psi= \Phi^{-1}$ \ and \ $\Phi$ is an isomorphism \ from algebra \ $L_{0}(\nabla_1)$ \ onto algebra \ $L_{0}(\nabla_2)$.
\end{proof}
We also need the following useful corollary from Theorem \ref{t32}
\begin{cor}\label{c34}.
Let $U(f) = U(\mathbf 1_{\nabla_1})\cdot \Phi(f)$ be a linear isometry from $L_{log}(\nabla_1, \mu_1)$ \ into \ $L_{\log}(\nabla_2, \mu_2)$ (see Theorem \ref{t32}), let $0\neq e\in  \nabla_1, \ q= \Phi(e)\in \nabla_2$. Then $U_e(f) = U(f), \ f \in L_{log}(e\cdot\nabla_1, \mu_1)$, is a linear isometry from $L_{log}(e\cdot\nabla_1, \mu_1)$ \ into \ $L_{\log}(q\cdot\nabla_2, \mu_2)$.
\end{cor}
\begin{proof}
It is clear that
$$U_e(f) =\Phi(e)\cdot U(\mathbf 1_{\nabla_1})\cdot \Phi(f)=q\cdot U(\mathbf 1_{\nabla_1})\cdot \Phi(f), \ f \in L_{log}(e\cdot\nabla_1, \mu_1)
$$
is a linear map from $L_{log}(e\cdot\nabla_1, \mu_1)$ \ into  \ $L_{\log}(\nabla_2, \mu_2)$.

If $f=e\cdot f \in L_{log}(e\cdot\nabla_1, \mu_1)\subset L_{log}(\nabla_1, \mu_1)$ then
$$
U(f)\cdot (\Phi(\mathbf 1_{\nabla_1}) - q)= U(\mathbf 1_{\nabla_1})\cdot \Phi(e\cdot f)\cdot (\Phi(\mathbf 1_{\nabla_1}) - \Phi(e))=
$$
$$
=U(\mathbf 1_{\nabla_1})\cdot \Phi(f)\cdot (\Phi(e\cdot (\mathbf 1_{\nabla_1}-e))=0.
$$
Consequently, $U_e(f) \in L_{\log}(q\cdot\nabla_2, \mu_2)$ for all $f \in L_{log}(e\cdot\nabla_1, \mu_1)$. In addition,
$$
\|U_e(f)\|_{\log,\mu_2} = \|U(f)\cdot \Phi(e)\|_{\log,\mu_2}=\|U(\mathbf 1_{\nabla_1})\cdot \Phi(f)\cdot \Phi(e)\|_{\log,\mu_2}=
$$
$$
=\|q\cdot U(\mathbf 1_{\nabla_1})\cdot \Phi(f)\|_{\log,\mu_2}=\int\limits_{\nabla_2}  \log\big(1+ q\cdot |U(f)|) \  d \mu_2=\int\limits_{\nabla_2}  \log\big(1+|U(f)|) \  d \mu_2=
$$
$$
=\|U(f)\|_{\log,\mu_2}=\|f\|_{\log,\mu_1}
$$
for all $f \in L_{log}(e\cdot\nabla_1, \mu_1)$.
\end{proof}

Let  $\nabla=\nabla_1=\nabla_2$, \  $\mu= \mu_1$, \ $\nu=\mu_2$, \ $h=\frac{d\nu}{d\mu}$.
The following Theorem gives a sufficient condition for the non-isometry of the $F$-spaces $L_{\log}(\nabla, \mu)$ and $ L_{\log}(\nabla, \nu)$.
\begin{teo} \label{t35} Let $\nabla$ be a complete Boolean algebra, let $\mu$ and $\nu$ be a strictly positive finite measures on $\nabla$, and let $\mu(\mathbf 1_{\nabla})\neq \nu(\mathbf 1_{\nabla})$. Then
$F$-spaces $L_{\log}(\nabla,\mu)$ and $ L_{\log}(\nabla, \nu)$ are
not isometric.
\end{teo}
\begin{proof}
Setting $h=\frac{d\nu}{d\mu}$, \ $t=\frac{\nu(\mathbf 1_{\nabla})}{\mu(\mathbf 1_{\nabla})}$, \ we have that $\int \limits_{\nabla}{h} d\mu =\nu(\mathbf 1_{\nabla})=t\cdot \mu(\mathbf 1_{\nabla})$, and
$$
\int \limits_{\nabla}\log(1+ \lambda)^{h} d\mu= \log(1+\lambda)\int
\limits_{\nabla}h d\mu=
$$
$$=\log(1+\lambda)\cdot t\cdot \mu(\mathbf 1_{\nabla})=\int
\limits_{\nabla}\log(1+ \lambda)^{t}d\mu. \eqno(4)
$$
for all $\lambda>0$.

Suppose that there exists a linear  surjective isometry $U$ from $L_{\log}(\nabla, \nu)$ onto $L_{\log}(\nabla, \mu)$.
Consider the following two cases.

$1)$. If $t > 1$, then for $ \lambda>1$  \  we have
$$
\int \limits_{\nabla}\log\frac{(1 +
\lambda)^{t}}{\lambda}d\mu>\int \limits_{\nabla}\log \frac{(1 +
\lambda)^{t}}{1+\lambda}d\mu> \|U(\mathbf 1_{\nabla})\|_{\log,\mu}.
$$
Then
$$
\int \limits_{\nabla}\log(1 + \lambda)^{t}d\mu-\int
\limits_{\nabla}\log \lambda \ d\mu>\|U(\mathbf 1_{\nabla}\|_{\log,\mu}
$$
and
$$
\int \limits_{\nabla}\log(1 + \lambda)^{t}d\mu>\int
\limits_{\nabla}\log \lambda \ d\mu+\|U(\mathbf 1_{\nabla})\|_{\log,\mu}=
$$
$$
=\int \limits_{\nabla}\log \lambda \ d\mu+\int\limits_{\nabla}\log (1+|U(\mathbf 1_{\nabla}))|d\mu>
\int\limits_{\nabla}\log \lambda \ d\mu+
$$
$$+\int\limits_{\nabla}\log(\frac{1}{\lambda}+|U(\mathbf 1_{\nabla})|)d\mu.
\eqno(5)
$$

Using $(4)$ and $(5)$, we get
$$
\|\lambda\cdot\mathbf 1_{\nabla}\|_{\log,\nu}=\int
\limits_{\nabla}\log(1+ \lambda\cdot\mathbf 1_{\nabla})d\nu=\int
\limits_{\nabla}h\cdot\log(1+ \lambda\cdot\mathbf 1_{\nabla} )d\mu=
$$
$$
=\int\limits_{\nabla}\log(1+ \lambda\cdot\mathbf 1_{\nabla} )^{h}d\mu=\int
\limits_{\nabla}\log( 1+ \lambda )^{t}d\mu
>\int \limits_{\nabla}\log\lambda \ d\mu+\int \limits_{\nabla}
\log(\frac{1}{\lambda}+ |U(\mathbf 1_{\nabla})|)d\mu=
$$
$$
+\int \limits_{\nabla}
\log(1+ |U(\lambda \cdot \mathbf 1_{\nabla})|)d\mu=\|U(\lambda \cdot
\mathbf 1_{\nabla})\|_{\log,\mu}.
$$
Consequently,
$$
\|\lambda \cdot \mathbf 1_{\nabla} 1\|_{\log,\nu}\neq\|U(\lambda \cdot
\mathbf 1_{\nabla})\|_{\log,\mu},
$$
that is, the map $U: L_{\log}(\nabla, \nu) \to L_{\log}(\nabla,
\mu)$ is not  isometry.

$2)$. If $0<t < 1$ then using the equalities $\frac{d\mu}{d\nu} =
(\frac{d\nu}{d\mu})^{-1} =h^{-1}$ we have
$$
\frac{\int \limits_{\nabla}h^{-1} d\nu}{\nu(\mathbf 1_{\nabla})}= \frac{\int
\limits_{\nabla}h\cdot h^{-1} d\mu}{\int
\limits_{\nabla}h d\mu}=\frac{\mu(\mathbf 1_{\nabla})}{\int
\limits_{\nabla}h d\mu}=\frac{1}{t}> 1.
$$

Using now the case $1)$ for the isometry $U^{-1}:
L_{\log}(\nabla, \mu) \to L_{\log}(\nabla,\nu),$ we get
$$
\|\lambda \cdot \mathbf 1_{\nabla}\|_{\log,\mu}\neq\|U^{-1}(\lambda \cdot
\mathbf 1_{\nabla})\|_{\log,\nu},
$$
that is, the map $U^{-1}$ is not  isometry.
\end{proof}
Now we can refine Theorem \ref{t32}.
\begin{teo}\label{t36}
Let $U:L_{log}(\nabla_1, \mu_1) \to L_{\log}(\nabla_2, \mu_2)$ be a  linear isometry. Then there exists an injective $\sigma$-additive homomorphism $\Phi$ \ from algebra \ $L_{0}(\nabla_1)$ \ into algebra \ $L_{0}(\nabla_2)$  \ such that \  $U(f) = U(\mathbf 1_{\nabla_1})\cdot \Phi(f)$ for each $f\in L_{\log}(\nabla_1, \mu_1)$, in addition, $\mu_2(\Phi(e)) = \mu_1(e)$ \ for all $e \in \nabla_1$.
\end{teo}
\begin{proof}
Let there exists $0\neq e \in \nabla_1$ such that $\mu_2(\Phi(e)) \neq \mu_1(e)$, that is,  $$\mu_1(e\cdot \mathbf 1_{\nabla_{1}})\neq \mu_2(\Phi(e)) \cdot \mathbf 1_{\nabla_{2}}).$$
By corollary \ref{c34}, \ the map \ $U_e$ \ is a linear surjective isometry from $L_{log}(e\cdot\nabla_1, \mu_1)$ \ onto \ $L_{\log}(\Phi(e)\cdot \Phi(\nabla_2), \mu_2)$, which is impossible by Theorem \ref{t35}.
\end{proof}

Using Theorems \ref{t22}, \ref{t36} and Corollary \ref{c33}, we have the following  criterion for isometric of the $F$-spaces $L_{log}(\nabla_1, \mu_1)$ and  $L_{\log}(\nabla_2, \mu_2)$.

\begin{teo}\label{t47}
The $F$-spaces $L_{log}(\nabla_1, \mu_1)$ and  $L_{\log}(\nabla_2, \mu_2)$ are isometric if and only if there exists a preserves  measure isomorphism $\varphi:(\nabla_1, \mu_1) \to (\nabla_2, \mu_2)$.
\end{teo}

Now we give criterion of isometric for the $F$-spaces $L_{log}(\nabla_1, \mu_1)$ and  $L_{\log}(\nabla_2, \mu_2)$ using the  passports of the Boolean algebras $\nabla_1$ and $\nabla_2$.

Let $ \nabla$ be a non-atomic complete Boolean algebra and let $\mu$ be  a strictly positive finite measure on $ \nabla$.  By $\tau(e\cdot \nabla)$ denote the minimal cardinality of a set that is dense in the Boolean algebra $ e\cdot \nabla$ with respect to the order topology ($(o)$-topology). A  non-atomic complete Boolean algebra $ \nabla$ is said to be homogeneous if $\tau(e\cdot \nabla)=\tau(g\cdot \nabla)$ for any nonzero $ e, g \in \nabla$. The cardinality $\tau(\nabla)$ is called the weight of a homogeneous Boolean algebra $\nabla$ (see, for example \cite[Ch. VII]{V}).

Since $\mu(\mathbf 1_{\nabla})<\infty$ it follows that  $\nabla$ is  a  direct product of homogeneous Boolean algebras $e_n\cdot \nabla$, \ where \ $e_n\cdot e_m = 0, \ n\neq m, \ \tau_n=\tau (e_n\cdot\nabla)< \tau_{n+1}, \ n,m =1,2,\dots$ \ \cite[Ch. VII, \S \ 2]{V}).

Set $\alpha_n = \mu(e_n)$. The matrix
$\left(
                  \begin{array}{ccc}
                    \tau_{1} & \tau_{2} & \dots \\
                    \alpha_1 & \alpha_2 & \dots \\
                  \end{array}
                \right)$
is called  {\it the  passport of  Boolean algebra $(\nabla, \mu)$}.

The following theorem gives a classification of Boolean algebras with finite measures \cite[Ch. VII, \S \ 2]{V}.
\begin{teo}\label{t48}  Let $\mu_i$ be  a strictly positive finite measure  on the  non-atomic complete Boolean algebra  $\nabla_i$, \ and let $\left(
\begin{array}{ccc}
                   \tau_{1}^{(i)} & \tau_{2}^{(i)} & \dots  \\
                    \alpha_1^{(i)} &\alpha_2^{(i)} & \dots \\
                  \end{array}
                \right)$ be the passport of  $(\nabla_i, \mu_i)$, \ $i=1, 2$.
Then the following conditions are equivalent:

$(i)$. There exists a preserves  measure isomorphism  $\varphi:(\nabla_1, \mu_1) \to (\nabla_2, \mu_2)$;

$(ii)$.  $\tau_{n}^{(1)} =\tau_{n}^{(2)}$ and  $\alpha_n^{(1)}  =\alpha_n^{(2)} $  for all $n=1,2\dots$
\end{teo}
Using Theorems \ref{t47} and \ref{t48} we obtain the following  criterion for isometric of the $F$-spaces $L_{log}(\nabla_1, \mu_1)$ and  $L_{\log}(\nabla_2, \mu_2)$.
\begin{cor}\label{c37}  Let $(\nabla_i, \mu_i)$ be   the same as in Theorem \ref{t48}. Then  the $F$-spaces $L_{log}(\nabla_1, \mu_1)$ and  $L_{\log}(\nabla_2, \mu_2)$ are isometric if and only if \ $\tau_{n}^{(1)} =\tau_{n}^{(2)}$ and   $\alpha_n^{(1)}  =\alpha_n^{(2)} $  for all $n=1,2\dots$, where $\left(
\begin{array}{ccc}
                   \tau_{1}^{(i)} & \tau_{2}^{(i)} & \dots  \\
                    \alpha_1^{(i)} &\alpha_2^{(i)} & \dots \\
                  \end{array}
                \right)$ is the passport of Boolean algebra  $(\nabla_i, \mu_i)$.
\end{cor}

\begin{cor}\label{c38} (cf. Theorem \ref{t35}). Let $\nabla$ be a homogeneous Boolean algebra and let  $\mu, \ \nu$ \ be  a strictly positive finite measures  on   $\nabla$. Then  the $F$-spaces $L_{log}(\nabla, \mu)$ \ and \ $L_{\log}(\nabla, \nu)$ \ are isometric if and only if \ $\mu(\mathbf 1_{\nabla}) =\nu(\mathbf 1_{\nabla})$.
\end{cor}


\begin{thebibliography}{99}


\bibitem{ach}  R.Z. Abdullaev, V.I. Chilin,  \emph{Isomorphic Classification of $\ast$-Algebras of Log-Integrable Measurable Functions}. Algebra, Complex Analysis, and Pluripotential Theory. USUZCAMP 2017. Springer Proceedings in Mathematics and Statistics, Vol. 264, 73--83.

\bibitem{B}  S. Banach,  \emph{Theorie des operations lineaires.} Warsaw, 1932.

\bibitem{chl} V. Chilin, S. Litvinov, \emph{The validity space of Dunford-Schwartz pointwise ergodic theorem}. J. Math. Anal. Appl. Vol. 461, 2018, 234-–247.

\bibitem{DSZ} K. Dykema, F. Sukochev, D. Zanin, \emph{Algebras of  log-integrable functions and operators}. Complex Anal. Oper. Theory, Vol. 10 (8), 2016, 1775--1787.

\bibitem{HSZ} J. Huang, F. Sukochev, D. Zanin, \emph{Logarithmic submajorisation and order-preserving linear isometries}. ArXiv:1808.10557v2 [math.FA] 20 Nov 2018, 42 p.

\bibitem{FJ}  R. Fleming,  J.E. Jamison, \emph{Isometries on Banach spaces: function spaces}. CRC Press Company, Boca Raton London NewYork Washington, D.C., 2003.

\bibitem{K}  N.J.Kalton, N.T Peck, \emph{An $F$-space sampler}. London Math. Society lecture series,  Cambridge University Press, 1984.

\bibitem{L} J. Lamperti, \emph{On the isometries of some function spaces}, Pacific J. Math.,  V. 8, 1958,  459--466.
\bibitem{V} D.A. Vladimirov, Boolean Algebras in Analysis. Mathematics and its Applications, 540, Kluwer Academic Publishers, Dordrecht, 2002.

\end{thebibliography}
\end{document}